\documentclass[11pt]{amsart}   	
\usepackage{geometry}             
\usepackage[dvipsnames]{xcolor}   		             		
\usepackage{graphicx}				
\usepackage{amssymb,mathrsfs}
\usepackage{amsthm,amsmath,stmaryrd}
\usepackage{tikz}
\usepackage{tikz-cd}
\usepackage{accents,upgreek,enumerate}
\usepackage[headings]{fullpage}
\usepackage{bm}
\usepackage{mathtools}
\usepackage[all]{xy}
\usepackage{caption}

\usepackage[euler-digits]{eulervm}
\usepackage[normalem]{ulem}

\tikzset{
  commutative diagrams/.cd, 
  arrow style=tikz, 
  diagrams={>=stealth}
}

\newtheorem{innercustomthm}{Theorem}
\newenvironment{customthm}[1]
  {\renewcommand\theinnercustomthm{#1}\innercustomthm}
  {\endinnercustomthm}

  \makeatletter
\def\@tocline#1#2#3#4#5#6#7{\relax
  \ifnum #1>\c@tocdepth % then omit
  \else
    \par \addpenalty\@secpenalty\addvspace{#2}%
    \begingroup \hyphenpenalty\@M
    \@ifempty{#4}{%
      \@tempdima\csname r@tocindent\number#1\endcsname\relax
    }{%
      \@tempdima#4\relax
    }%
    \parindent\z@ \leftskip#3\relax \advance\leftskip\@tempdima\relax
    \rightskip\@pnumwidth plus4em \parfillskip-\@pnumwidth
    #5\leavevmode\hskip-\@tempdima
      \ifcase #1
       \or\or \hskip 1em \or \hskip 2em \else \hskip 3em \fi%
      #6\nobreak\relax
    \dotfill\hbox to\@pnumwidth{\@tocpagenum{#7}}\par
    \nobreak
    \endgroup
  \fi}
\makeatother

\usetikzlibrary{calc}
\usetikzlibrary{fadings}
\usetikzlibrary{decorations.pathmorphing}
\usetikzlibrary{decorations.pathreplacing}

\newcounter{marginnote}
\DeclareMathAlphabet{\mathpzc}{OT1}{pzc}{m}{it}

\usepackage[backref=page]{hyperref}
\hypersetup{
  colorlinks   = true,          %Colors links instead of ugly boxes
  urlcolor     = purple,          %Color for external hyperlinks
  linkcolor    = blue,          %Color of internal links
  citecolor   = violet             %Color of citations
}

\newtheorem{theorem}{Theorem}[subsection]

\newtheorem{lemma}[theorem]{Lemma}
\newtheorem{proposition}[theorem]{Proposition}

\newtheorem{quasi-theorem}[theorem]{Quasi-Theorem}

\theoremstyle{definition}
\newtheorem{definition}[theorem]{Definition}

\newtheorem{example}[theorem]{Example}
\newtheorem{blank remark}[theorem]{}

\theoremstyle{remark}
\newtheorem{rem1}[theorem]{Remark}
\newenvironment{remark}{\begin{rem1}\em}{\end{rem1}}

\newtheorem{not1}[theorem]{Notation}

\newcommand{\QQ} {{\mathbb Q}}		
\newcommand{\RR} {{\mathbb R}}

\DeclareMathOperator{\spec}{Spec}

\newcommand{\cal}{\mathcal}

\def\cM{{\cal M}}

\def\fM{\mathfrak{M}}

\newcommand{\plC}{\scalebox{0.8}[1.3]{$\sqsubset$}}

\newcommand{\Mbar}{\overline{\cM}\vphantom{\cM}}

\makeatletter
\def\blfootnote{\xdef\@thefnmark{}\@footnotetext}
\makeatother

\title{A note on the cycle of curves in a product of pairs}

\date{}

\author{Dhruv Ranganathan}

\address{Department of Pure Mathematics {\it \&} Mathematical Statistics\\
University of Cambridge, UK}
\email{\href{mailto:dr508@cam.ac.uk}{dr508@cam.ac.uk}}

\begin{document}

\maketitle

\begin{abstract}
We prove that the cycle-valued logarithmic Gromov--Witten theory of a product of simple normal crossings pairs $X\times Y$ decomposes into a product of pieces coming from $X$ and $Y$, provided that the decomposition is considered over a blowup of the moduli space of curves. The result is established by applying the toroidal semistable reduction theorem to the stabilization morphism on the stack of maps to the Artin fans. Since products of smooth pairs are naturally simple normal crossings pairs, the result provides a direct avenue of access to questions in logarithmic Gromov--Witten theory via relative Gromov--Witten theory. 
\end{abstract}

%\eject

\section{Introduction}

Logarithmic Gromov--Witten theory concerns cycles on moduli of maps
\[
(C,p_1,\ldots, p_n)\to (Z,D_Z)
\] 
from pointed curves to a simple normal crossings pair. The logarithmic structure on such a morphism prescribes the contact orders of each point $p_i$ along each divisor $D_j$. The contact order is equal to the scheme theoretic tangency order for non-degenerate maps, but is well-defined for arbitrary logarithmic maps and remains locally constant in families. The resulting moduli space $\mathsf K_\Gamma(Z)$ is equipped with a virtual fundamental class. Enumerative invariants are integrals of evaluation cycles from $Z$ and tautological cohomology classes against the virtual class~\cite{AC11,Che10,GS13}. When the divisor $D_Z$ is smooth, the theory is equivalent to Li's relative Gromov--Witten theory~\cite{Li01}. The latter is a subject in which researchers have discovered structural results which remain mysterious in logarithmic Gromov--Witten theory, see for instance~\cite{FP,FWY20,GV05,JPPZ2,MP06}. 

Our interest in this paper is the logarithmic Gromov--Witten theory of products of pairs. Products of nontrivial smooth pairs are simple normal crossings but never smooth, and we view such products as a prime testing ground for the structure of the logarithmic theory. %Second, as logarithmic Gromov--Witten theory is insensitive to birational modifications, a variety need only be birational to a product for the result to apply~\cite{AW}. 

The contribution of this article is a product formula in logarithmic Gromov--Witten theory, generalizing a theorem of Behrend~\cite{Beh97}. The novelty in the paper is that the naive generalization of the result is demonstrably false, and we offer a geometric explanation of this failure in the main text and explain how to correct it. %Finally, we propose a conjecture about cycles in the moduli space of curves produced from logarithmic Gromov--Witten theory, and provide evidence for it in the case of products. 

\subsection{Product rules} Let $(X,D_X)$ and $(Y,D_Y)$ be smooth projective varieties equipped with simple normal crossings divisors, and assume that all non-empty intersections of components of $D_X$ and of $D_Y$ are connected. The product is a simple normal crossings pair:
\[
Z = X\times Y, \ \ \ \ D_Z = \pi_X^{-1}(D_X)\cup\pi_Y^{-1}(D_Y)
\]
where $\pi_X$ and $\pi_Y$ are the two natural projections. Enumerate the irreducible logarithmic divisors in $Z$ by $D_1,\ldots, D_r$ in $Z$. We fix discrete data for maps from curves to $Z$: 
\begin{enumerate}[(i)]
\item the genus of the source curve $g$,
\item the curve class $\beta$ in the homology group $H_2(Z)$,
\item the number $n$ of labelled marked points, which will be denoted $p_1,\ldots, p_n$,
\item the positive integers $c_{ij}$ giving the contact order of the point $p_i$ with the divisor $D_j$.
\end{enumerate}
The entirety of the discrete data will be packaged in the symbol $\Gamma$. The curve class $\beta$ determines compatible classes on each factor by the K\"unneth formula. The components of the divisors on $Z$ are pulled back along one of the two projections, so for each $p_i$ the integers $c_{ij}$ determine compatible contact orders for maps to $X$ and $Y$. We obtain compatible discrete data for maps to $X$ and maps to $Y$. There is a moduli problem on logarithmic schemes parameterizing families of logarithmic maps from logarithmic curves to a fixed logarithmic scheme, whose underlying morphisms are stable in the traditional sense. The moduli problem is representable by a proper algebraic stack equipped with a logarithmic structure and a virtual fundamental class~\cite{AC11,Che10,GS13}. 

We assume that the genus and marking numbers lie in the stable range. As we explain in the main text, for logarithmic mapping stacks $\mathsf K$, there is a simple construction that produces logarithmic modifications
\[
\mathsf K^\dagger\to \mathsf K,
\]
together with a virtual class on $\mathsf K^\dagger$ that pushes down to the virtual class $\mathsf K$. We note that the logarithmic modifications considered here are obtained by pulling back birational modifications along a map from $\mathsf K$ to a smooth space, and in particular, $\mathsf K^\dagger$ need not actually be birational to $\mathsf K$. It is best thought of as a ``virtual'' birational model. 

Our main result explains that the product rule in logarithmic Gromov--Witten theory holds provided the moduli spaces are replaced by appropriate virtual birational models. Without passing to such modifications, the product rule fails; see the discussion after the theorem statement. We maintain the notation above: the target pairs are $X$ and $Y$, their product is $Z$, and the discrete data $\Gamma$ is fixed. 

\begin{customthm}{A}\label{thm: product-formula}
There exists an explicit logarithmic modification of the moduli space of stable curves
\[
\begin{tikzcd} 
\Mbar_{\Gamma}\arrow{r}{\pi} &\Mbar_{g,n}
\end{tikzcd} 
\]
and virtual birational models $\mathsf K_\Gamma(W)^\dagger\to\mathsf K_\Gamma(W)$ for each $W = X,Y,Z$, fitting into a diagram
\[
\begin{tikzcd}
\mathsf K_\Gamma(Z)^\dagger\arrow{r}{\vartheta} & \mathsf P_\Gamma(X\times Y) \arrow{d}\arrow{r}\arrow[dr, phantom, "\square"] & \mathsf K_\Gamma(X)^\dagger\times \mathsf K_\Gamma(Y)^\dagger \arrow{d} \\
&\Mbar_{\Gamma}\arrow{r}[swap]{\Delta} &\Mbar_{\Gamma}\times \Mbar_{\Gamma}. & 
\end{tikzcd}
\]
There is an equality of virtual classes
\[
\vartheta_\star [\mathsf K_\Gamma(Z)^\dagger]^{\mathrm{vir}} = \Delta^![\mathsf K_\Gamma(X)^\dagger\times \mathsf K_\Gamma(Y)^\dagger]^{\mathrm{vir}} \ \ \mathrm{in} \ \ A_\star(\mathsf P_\Gamma(X\times Y);\QQ).
\]
\end{customthm}

\subsection{Discussion and context} The product formula in ordinary Gromov--Witten theory is a basic property, analogous to a K\"unneth formula, proved by Behrend in a paper that shortly followed the foundations of the virtual fundamental class in algebraic geometry~\cite{Beh97}. The result has had applications, for example to the enumerative geometry of compact Calabi--Yau geometries~\cite{OP18}, and forms a basic exploratory tool in the subject. For example, the Gromov--Witten theory of target curves is known to be tautological~\cite{Jan17}, and the product formula immediately exhibits an infinite class of algebraic surfaces for which the same is true. It is this flavour of result that forms the motivation for the theorem here. Indeed, together with the techniques developed in our follow up~\cite{MR21} and Janda's work on the Gromov--Witten theory of target curves~\cite{Jan17}, the product formula should yield the following result.

\noindent
{\bf Speculation.} {\it Let $X$ and $Y$ be smooth curves , each equipped with a logarithmically smooth structure. Then all logarithmic Gromov--Witten cycles of $X\times Y$ lie in the tautological ring of the moduli space of curves.}

The first exploration of a product formula in relative Gromov--Witten theory appears to be work of Lee and Qu, who prove that the simplest form of the product rules holds provided one factor in the product carries no logarithmic divisor, i.e. in the situation where the product is itself a smooth pair~\cite{LQ18}. Although restrictive, Lee and Qu's result is sharp -- if both factors have non-trivial logarithmic structure, the naive product formula typically fails. These failures are dramatic and commonplace, although it has taken some time for examples to appear in the literature.\footnote{An explicit counterexample was presented by the author in the working seminar on bChow intersections run by Holmes and Schmitt with notes available here~\url{https://johannesschmitt.gitlab.io/bChowSeminar.html}. A detailed explanation of this example is found in~\cite[Section~5.1]{CHL20}. The examples in~\cite[Section~1]{NR19} can also be adapted to this setting.}

The point that we wish to illustrate in this paper is that the failure of the product formula can be traced to the stabilization morphism
\[
\mathsf K_\Gamma(X)\to \Mbar_{g,n}.
\]
If $X$ has trivial logarithmic structure the stabilization is ``virtually flat''. Once $X$ acquires logarithmic structure, the map is only ``virtually logarithmically smooth'', and not virtually flat.\footnote{These notions are not standard in the literature but here we simply mean a morphism that factors as a strict logarithmic morphism with a perfect relative obstruction theory following by a flat resp. logarithmically smooth map.} The failure can be patched by a semistable reduction procedure, which restores the product rule. 

\subsection{Nearby mathematics} The results of this paper were understood contemporaneously with the author's work on the degeneration formula~\cite{R19}, where virtually flat tautological morphisms become poorly behaved logarithmically. The present paper was initially written in 2019 as a simpler companion paper to loc. cit. that demonstrated parallel ideas. It has undergone expository restructuring in the years since. 

A logarithmic product rule has been proved by Herr~\cite{Herr}. The approach is complementary; Herr's work establishes a statement on the moduli spaces themselves rather than their blowups, but does not equate cycles in ordinary Chow homology. In short, we change the moduli spaces, while he changes the intersection theory. We believe Herr's formalism has conceptual value, but the present result is more effective for numerical consequences. The strong product decomposition for rational curves in the logarithmic torus has a similar nature to Herr's work~\cite{RW19}.

When specializing to genus $0$ curves in convex targets, there is a ``relative'' version of the product formula, relating the pairs $(X,D)$ and $(X,E)$ with $(X,D+E)$ by studying the fiber product of the moduli space of maps to $X$. This was examined and utilized in work with Nabijou based on the strategy outlined here~\cite{NR19}. Just as the naive product formula fails, this relative version also fails. The failure lies at the heart of counterexamples to the local/logarithmic conjectures, while the blowup geometry leads to an appropriately correction version~\cite{vGGR}. The geometry in the work with Nabijou primarily concerns the ``maximal contact'' geometry, and the virtual birational models constructed there are significantly more explicit than we produce here. 

\subsection{Recent progress} When the target geometry is $(\mathbb P^1)^r$, but as a ``rubber target'' for the rank $r$ torus action, the corrected product rule on a blowup was obtained using Abel--Jacobi theory by Holmes--Pixton--Schmitt~\cite{HPS19}. The cycles in $\Mbar_{g,n}$ obtained from these constructions -- dubbed the toric contact cycles -- have also seen recent interest. Independently, Holmes and Schwarz and Molcho and the author have shown that these cycles lie in the tautological ring~\cite{HS21,MR21}. Our work with Molcho builds on the strategy laid out in the first ar$\chi$iv version of this paper -- to understand the virtual strict transform operation to access the product rule. A brute force calculation scheme for handling virtual strict transforms is outlined in~\cite{MR21}. The picture was considerably refined in~\cite{HMPPS}. 

Finally, there has been significant recent interest in approaching the enumerative geometry of logarithmic pairs using roots rather than logarithmic structures~\cite{TY20}. The product formula is known to hold in the Gromov--Witten theory of stacks~\cite{AJT}, and this can be taken to be a simple explanation for why the logarithmic and orbifold theories do not coincide. In genus $0$, the situation is now well-understood, and product rules again play a fundamental role~\cite{BNR2}.

\subsection*{Acknowledgements}  My view on product geometries was shaped by discussions with friends and colleagues including D. Abramovich, T. Graber, A. Gross, D. Maulik, S. Molcho, N. Nabijou, and R. Pandharipande. I am especially grateful to D. Abramovich for the comment ``it would all be fine if the stabilization morphism was flat, but it probably fails in general'', which clarified the geometry. I thank H. Wise for keeping me informed of the work of L. Herr, and to D. Holmes and J. Schmitt for comments on a draft and for the opportunity to present the work at their seminar on bChow Intersections in 2020.

\noindent
{\bf Funding.} The author is supported by EPSRC grant EP/V051830/1. 

\section{Logarithmic GW theory}

We work over the complex numbers throughout; the logarithmic schemes appearing will all be fine and saturated. 

\subsection{Target geometry} Let $W$ be a simple normal crossings pair with divisor $D_W$. Denote the divisor components by $D_1,\ldots, D_r$. Assume that every intersection of a subset of the divisor components of $D_W$ is connected. Each divisor component determines a tautological map 
\[
W\to [\mathbb A^1/\mathbb G_m]:=\mathsf A,
\]
since a map to the codomain is determined by a line bundle and a section. Ranging over the components, we obtain a map
\[
W\to \mathsf A^r.
\]
The target may also be equipped with a natural logarithmic structure using the simple normal crossings divisor obtained as the complement of the open point. Once it is so equipped, the map above becomes strict. 

\begin{remark}
If the divisor $D_W$ is merely normal crossings, or more generally toroidal, one may replace $\mathsf A^r$ with the Artin fan~\cite{AW}. However, the condition imposed here is always satisfied after strata blowups, and strata blowups do not affect logarithmic Gromov--Witten theory. 
\end{remark}

\subsection{Mapping spaces} We can consider the moduli space of logarithmic maps to $W$ and to $\mathsf A^r$. These are proven to be representable by algebraic stacks in~\cite{AC11,GS13} for logarithmically smooth targets such as $W$; the algebraicity of logarithmic maps to $\mathsf A^r$ is established in~\cite{AW}. 

We fix numerical data $\Gamma$ of genus, number of markings, contact orders, and in the case of $W$, the curve class. The map $W\to \mathsf A^r$ induces a strict map on the mapping spaces:
\[
\mathsf K_\Gamma(W)\to \mathsf K_\Gamma(\mathsf A^r). 
\]
We impose a stability condition on the left, i.e. that the underlying maps to $W$ are stable. No such restriction is sensible on the right. However, we work with the non-separated cover of $\mathsf K_\Gamma(\mathsf A^r)$ with decorations by curve class, i.e. it formally includes the data of a curve class in the monoid $H^+_2(W)$ of effective curve classes in $W$, one for each component, adding to the fixed class recorded by $\Gamma$. Details of these decorated spaces may be found for instance in~\cite[Section~2.2]{HL10}. The construction there exhibits a space $\mathfrak M^{\mathsf{wt}}$ as a non-separated cover over $\mathfrak M_g$, and we pullback the stack of maps to $\mathsf A^r$ along this cover to form $\mathsf K_\Gamma(\mathsf A^r)$.

The fundamental structure in logarithmic Gromov--Witten theory is the following. 

\begin{theorem}
The morphism
\[
\mathsf K_\Gamma(W)\to \mathsf K_\Gamma(\mathsf A^r). 
\]
is equipped with a relative perfect obstruction theory given by $R^\bullet \pi_\star F^\star T_W^{\mathsf{log}}$, where $\pi$ and $F$ denote the universal curve and universal map respectively. 
\end{theorem}

A theorem equivalent to the one above is proved by Gross and Siebert~\cite[Section~5]{GS13}, and the equivalence to the form above is established in~\cite[Section~6]{AW}. Note that if $D$ is empty, the target of morphism above is $\mathfrak M_{g,n}$ of curves. 

\subsection{Products in different categories} Let $X$ and $Y$ be simple normal crossings pairs and let $Z = X\times Y$. The stack of logarithmic maps (respectively to $Z$, $X$, or $Y$) represents the stack over logarithmic schemes whose fiber over a logarithmic scheme $S$ is the groupoid of logarithmic maps (respectively to $Z$, $X$, or $Y$) from families of logarithmic curves over $S$. As a formal consequence, the following commutative diagram of stacks over the category of fine and saturated logarithmic schemes is a fiber product:
\[
\begin{tikzcd}
\mathsf K(Z)\arrow{d}\arrow{r} & \mathsf K(X)\arrow{d}\\
\mathsf K(Y)\arrow{r} & \fM_{g,n}.
\end{tikzcd}
\]
See~\cite[Section~2]{AC11} for a discussion. A scheme or stack with a fine and saturated logarithmic structure determines a functor/fibered category on the category logarithmic schemes. Each of the fibered categories in the above diagram is represented, in this sense, by a stack with a logarithmic structure, it \textit{does not} follow that the fiber product of $\mathsf K(X)$ and $\mathsf K(Y)$ over $\fM_{g,n}$ as ordinary stacks is equal to $\mathsf K(Z)$. 

The issues are also present over the moduli space of stable curves. If $3g-3+n$ is non-negative, we can examine the fiber products
\[
\mathsf K(X)\times_{\Mbar_{g,n}}\mathsf K(Y) \ \ \ \textsf{vs.} \ \ \ \ \left(\mathsf K(X)\times_{\Mbar_{g,n}}\mathsf K(Y)\right)^{\mathsf{fs}},
\]
where the right hand side is the fine and saturated fiber product. A priori, there is no reason for these to coincide, and indeed they typically do not. We capture the failure of the product formula as coming from the following basic fact:
\begin{center}
\fbox{
\parbox{0.8\linewidth}{
\it The functor of maps to $X\times Y$ is given by a fine and saturated logarithmic fiber product  of $\mathsf K(X)$ and $\mathsf K(Y)$ over the moduli space of curves, but the intersection product calculates the scheme theoretic fiber product.}
}
\end{center}

When combined with tropical geometry, intersection theory may still be applied to calculate classes of fine and saturated fiber products, and a blueprint for this has been developed in~\cite{MR21}. 

\subsection{When do the two fiber products coincide?} Generally, given logarithmic schemes $X$, $Y$, and $B$, with maps $X\to B$ and $Y\to B$ there is a natural condition that guarantees that $X\times_B Y$ and $(X\times_B Y)^{\mathsf{fs}}$ coincide. Specifically, this occurs if $X\to B$ is a \textit{saturated} morphism of logarithmic schemes in the sense of~\cite{Tsu19}. We do not need to go into the details of the definition, but the following gives a geometric source of saturated morphisms that will in fact suffice for our purposes here. 

\begin{example}
Let $X\to B$ be a logarithmically smooth morphism of logarithmically smooth schemes (or stacks). If the morphism is flat with reduced fibers then it is saturated, see~\cite[Section~2.1]{Mol16} for a proof. In particular, a dominant equivariant morphism of toric varieties with smooth target, reduced fibers of constant dimension is a saturated morphism with the standard toric logarithmic structures. Conversely, a saturated logarithmically smooth morphism of logarithmically smooth schemes necessarily has equidimensional fibers. 
\end{example}

\subsection{Forgetful maps and their pathologies} We have explained that ordinary and scheme theoretic fiber products of logarithmic schemes do not coincide and the relevance this has for mapping spaces to product targets. We now record an elementary and explicit example showing that forgetful morphisms are typically not saturated. 

The logarithmic mapping stack is equipped with a forgetful morphism to the space of curves:
\[
\mathsf K_\Gamma(W)\to \Mbar_{g,n}.
\]
The morphism above factors as
\[
\mathsf K_\Gamma(W)\to \mathsf K_\Gamma(\mathsf A^r)\to \fM_{g,n}\to \Mbar_{g,n}.
\]
The first morphism is strict, and so for the purposes of analyzing the difference between logarithmic and ordinary fiber products, it can essentially be ignored. 

\begin{lemma}
The stabilization morphism 
\[
\mathsf K_\Gamma(\mathsf A^r)\to \Mbar_{g,n}
\]
is logarithmically smooth. 
\end{lemma}

\begin{proof}
The morphism $\mathsf K_\Gamma(\mathsf A^r)\to \fM_{g,n}$ is shown to be logarithmically \'etale in~\cite[Section~3]{AW}. After passing to a cover of $\mathfrak M_{g,n}$ by stacks of stable curves with more marked points, the stabilization morphism from $\fM_{g,n}$ to $\Mbar_{g,n}$ is given locally by forgetting marked points. The forgetful morphisms are logarithmically smooth so the result follows. 
\end{proof}

As we will see, the product formula does not hold in logarithmic Gromov--Witten theory, but it does hold in ordinary Gromov--Witten theory. The reason for this can be traced to a basic phenomenon. In the notation above, the map
\[
\mathsf K_\Gamma(\mathsf A^r)\to \Mbar_{g,n}
\]
has dramatically different behaviour when $r = 0$ as compared to when $r$ is positive. 

\noindent
{\bf Key Observation.} If $r = 0$ the stabilization morphism 
\[
\mathsf K_\Gamma(\mathsf A^0) = \fM_{g,n} \to \Mbar_{g,n}
\]
is flat. When $r>0$ there exist discrete data $\Gamma$ such that the stabilization morphism need not be flat. Similarly, the forgetful morphism to $\fM_{g,n}$ has parallel issues: it is always logarithmically smooth but almost never flat. 

The observation is justified by the following example. We focus on the space of stable curves, but in fact the example below justifies the statement for both targets. 

\begin{example}
Consider the moduli stack $\mathsf K_\Gamma(\mathsf A)$, with discrete data as follows: the genus $g$ is taken to be $0$, the number of marked points is $5$ each with contact order $0$. Let $C$ be a genus $0$ curve over $\spec \mathbb C$. Consider logarithmic maps of the following form. The curve $C$ is a union of three irreducible curves
\[
C = C_1\cup C_0\cup C_2,
\]
where all components on the right are smooth rational curves, such that $C_0\cap C_1$ and $C_0\cap C_2$ are both nodes, and there are no other nodes. Distribute two points to $C_1$ and $C_2$ and one point to $C_0$ ensuring all components are stable. We can construct a map
\[
C\to \mathsf A
\]
sends $C_0$ to the $B\mathbb G_m$ point in $\mathsf A$, and the complement of the node for each of $C_1$ and $C_2$ maps to the complement of the $B\mathbb G_m$ point. The construction can be carried out as follows. A morphism
\[
C\to \mathsf A
\]
is a the data of a line bundle and section. If $C$ is a logarithmic curve over a $P$-logarithmic point for some monoid $P$, then there is an associated tropicalization
\[
\plC\to \sigma_P,
\]
where $P$ is the dual cone. The object $\plC$ is a cone complex and by elementary toric geometry, a piecewise linear function
\[
\plC \to\RR_{\geq 0}
\]
determines the data of a line bundle and section. A logarithmic map from $C$ to $\mathsf A$ is precisely a line bundle and section that arises in this way. So given a logarithmic curve $C$ over a $P$-logarithmic closed point, a logarithmic map is precisely a piecewise linear function on the tropicalization. In our example, we choose the piecewise linear function to be the following one:
\begin{figure}[h!]

\tikzset{every picture/.style={line width=0.75pt}} %set default line width to 0.75pt        

\begin{tikzpicture}[x=0.75pt,y=0.75pt,yscale=-1,xscale=1]
%uncomment if require: \path (0,653); %set diagram left start at 0, and has height of 653

%Straight Lines [id:da7376353213705786] 
\draw    (230,201) -- (377,200.02) ;
\draw [shift={(380,200)}, rotate = 179.62] [fill={rgb, 255:red, 0; green, 0; blue, 0 }  ][line width=0.08]  [draw opacity=0] (10.72,-5.15) -- (0,0) -- (10.72,5.15) -- (7.12,0) -- cycle    ;
\draw [shift={(230,201)}, rotate = 359.62] [color={rgb, 255:red, 0; green, 0; blue, 0 }  ][fill={rgb, 255:red, 0; green, 0; blue, 0 }  ][line width=0.75]      (0, 0) circle [x radius= 3.35, y radius= 3.35]   ;
%Straight Lines [id:da32084563819557377] 
\draw    (230,100) -- (310,130) ;
\draw [shift={(310,130)}, rotate = 20.56] [color={rgb, 255:red, 0; green, 0; blue, 0 }  ][fill={rgb, 255:red, 0; green, 0; blue, 0 }  ][line width=0.75]      (0, 0) circle [x radius= 3.35, y radius= 3.35]   ;
\draw [shift={(230,100)}, rotate = 20.56] [color={rgb, 255:red, 0; green, 0; blue, 0 }  ][fill={rgb, 255:red, 0; green, 0; blue, 0 }  ][line width=0.75]      (0, 0) circle [x radius= 3.35, y radius= 3.35]   ;
%Straight Lines [id:da9550828346903846] 
\draw    (230,150) -- (310,130) ;
\draw [shift={(310,130)}, rotate = 345.96] [color={rgb, 255:red, 0; green, 0; blue, 0 }  ][fill={rgb, 255:red, 0; green, 0; blue, 0 }  ][line width=0.75]      (0, 0) circle [x radius= 3.35, y radius= 3.35]   ;
\draw [shift={(230,150)}, rotate = 345.96] [color={rgb, 255:red, 0; green, 0; blue, 0 }  ][fill={rgb, 255:red, 0; green, 0; blue, 0 }  ][line width=0.75]      (0, 0) circle [x radius= 3.35, y radius= 3.35]   ;
%Straight Lines [id:da877308666521825] 
\draw    (420,130) -- (420,178) ;
\draw [shift={(420,180)}, rotate = 270] [color={rgb, 255:red, 0; green, 0; blue, 0 }  ][line width=0.75]    (10.93,-3.29) .. controls (6.95,-1.4) and (3.31,-0.3) .. (0,0) .. controls (3.31,0.3) and (6.95,1.4) .. (10.93,3.29)   ;
%Straight Lines [id:da939492699197416] 
\draw    (230,100) -- (220,80) ;
%Straight Lines [id:da9560865335938908] 
\draw    (230,150) -- (220,130) ;
%Straight Lines [id:da809585027314788] 
\draw    (310,130) -- (310,110) ;
%Straight Lines [id:da9346099606895686] 
\draw    (230,100) -- (240,80) ;
%Straight Lines [id:da1306224562368301] 
\draw    (230,150) -- (240,130) ;

% Text Node
\draw (411,102.4) node [anchor=north west][inner sep=0.75pt]    {$\ \plC$};
% Text Node
\draw (408,190.4) node [anchor=north west][inner sep=0.75pt]    {$\mathbb{R}_{\geq 0}$};
% Text Node
\draw (271,92.4) node [anchor=north west][inner sep=0.75pt]    {$1$};
% Text Node
\draw (272,143.4) node [anchor=north west][inner sep=0.75pt]    {$1$};
\end{tikzpicture}
\end{figure}
In the figure above, each edge maps onto the target with slope $1$. We claim that the locus parameterizing maps of this form in the space $\mathsf K_\Gamma(\mathsf A)$ is a divisor. Tropically, this can be seen simply from the fact that the tropical moduli here is only the choice of edge length for the two edges; the lengths have to be the same in order for the function to be continuous. Geometrically, we can see this by observing that although there are two nodes, the nodes must be smoothed simultaneously in any logarithmic deformation of the map. 

The divisorial locus here maps to a codimension $2$-stratum in the moduli space of stable curves, in this case $\Mbar_{0,5}$.  The morphism
\[
\mathsf K_\Gamma(\mathsf A)\to \Mbar_{0,5}
\]
is a logarithmically smooth morphism such that the preimage of a codimension $2$ stratum in the target is a divisor on the source. It follows that the morphism cannot be equidimensional.

The failure of flatness can be seen tropically. Indeed, if we remove all other divisors from $\mathsf K_\Gamma(\mathsf A)$ except for the one above, the cone complex of the resulting logarithmically smooth stack $\mathsf K_\Gamma^\circ(\mathsf A)$ is a single ray. The cone complex of $\Mbar_{0,5}$ is a $2$-dimensional fan containing a $2$-dimensional cone corresponding to curves of the type $\plC$ given above. The map on cone complexes is given in the figure below. One immediately sees the violation of flatness. Equidimensional logarithmically smooth morphisms have the property that in the induced map of cone complexes, every cone of the source maps surjectively onto a cone of the target, which is clearly violated here~\cite[Section~4]{AK00}.  
\begin{figure}[h!]
\tikzset{every picture/.style={line width=0.75pt}} %set default line width to 0.75pt        

\begin{tikzpicture}[x=0.75pt,y=0.75pt,yscale=-1,xscale=1]
%uncomment if require: \path (0,653); %set diagram left start at 0, and has height of 653

%Shape: Grid [id:dp0036653019811554444] 
\draw  [draw opacity=0][dash pattern={on 0.84pt off 2.51pt}][line width=0.75]  (347.37,162.35) -- (346.81,2.06) -- (507,1.5) -- (507.56,161.79) -- cycle ; \draw  [dash pattern={on 0.84pt off 2.51pt}][line width=0.75]  (347.37,162.35) -- (507.56,161.79)(347.3,142.35) -- (507.49,141.79)(347.23,122.35) -- (507.42,121.79)(347.16,102.35) -- (507.35,101.79)(347.09,82.35) -- (507.28,81.79)(347.02,62.35) -- (507.21,61.79)(346.95,42.35) -- (507.14,41.79)(346.88,22.35) -- (507.07,21.79)(346.81,2.35) -- (507,1.79) ; \draw  [dash pattern={on 0.84pt off 2.51pt}][line width=0.75]  (347.37,162.35) -- (346.81,2.06)(367.37,162.28) -- (366.81,1.99)(387.37,162.21) -- (386.81,1.92)(407.37,162.14) -- (406.81,1.85)(427.37,162.07) -- (426.81,1.78)(447.37,162) -- (446.81,1.71)(467.37,161.93) -- (466.81,1.64)(487.37,161.86) -- (486.81,1.57)(507.37,161.79) -- (506.81,1.5) ; \draw  [dash pattern={on 0.84pt off 2.51pt}][line width=0.75]   ;
%Straight Lines [id:da10318510584449658] 
\draw    (347.37,162.35) -- (346.82,5.35) ;
\draw [shift={(346.81,2.35)}, rotate = 89.8] [fill={rgb, 255:red, 0; green, 0; blue, 0 }  ][line width=0.08]  [draw opacity=0] (8.93,-4.29) -- (0,0) -- (8.93,4.29) -- cycle    ;
\draw [shift={(347.37,162.35)}, rotate = 269.8] [color={rgb, 255:red, 0; green, 0; blue, 0 }  ][fill={rgb, 255:red, 0; green, 0; blue, 0 }  ][line width=0.75]      (0, 0) circle [x radius= 3.35, y radius= 3.35]   ;
%Straight Lines [id:da2470550535556184] 
\draw    (347.37,162.35) -- (504.37,161.8) ;
\draw [shift={(507.37,161.79)}, rotate = 179.8] [fill={rgb, 255:red, 0; green, 0; blue, 0 }  ][line width=0.08]  [draw opacity=0] (8.93,-4.29) -- (0,0) -- (8.93,4.29) -- cycle    ;
%Straight Lines [id:da2450989612959067] 
\draw    (119.73,162.2) -- (226.88,25.56) ;
\draw [shift={(228.73,23.2)}, rotate = 128.1] [fill={rgb, 255:red, 0; green, 0; blue, 0 }  ][line width=0.08]  [draw opacity=0] (8.93,-4.29) -- (0,0) -- (8.93,4.29) -- cycle    ;
\draw [shift={(119.73,162.2)}, rotate = 308.1] [color={rgb, 255:red, 0; green, 0; blue, 0 }  ][fill={rgb, 255:red, 0; green, 0; blue, 0 }  ][line width=0.75]      (0, 0) circle [x radius= 3.35, y radius= 3.35]   ;
%Straight Lines [id:da5776418899937767] 
\draw  [dash pattern={on 3.75pt off 3pt on 7.5pt off 1.5pt}]  (347.37,162.35) -- (504.69,3.92) ;
\draw [shift={(506.81,1.79)}, rotate = 134.8] [fill={rgb, 255:red, 0; green, 0; blue, 0 }  ][line width=0.08]  [draw opacity=0] (8.93,-4.29) -- (0,0) -- (8.93,4.29) -- cycle    ;
%Curve Lines [id:da378780788377568] 
\draw    (226,89.5) .. controls (269.12,65.98) and (286.31,68.39) .. (323.69,88.26) ;
\draw [shift={(326,89.5)}, rotate = 208.3] [fill={rgb, 255:red, 0; green, 0; blue, 0 }  ][line width=0.08]  [draw opacity=0] (8.93,-4.29) -- (0,0) -- (8.93,4.29) -- cycle    ;

% Text Node
\draw (114,186.9) node [anchor=north west][inner sep=0.75pt]    {Cone complex of $\mathsf K^\circ_\Gamma(\mathsf A)$};
% Text Node
\draw (341,185.9) node [anchor=north west][inner sep=0.75pt]    {Cone complex of $\Mbar_{0,5}$};
\end{tikzpicture}
\caption{The induced morphism on cone complexes near the locus discussed above. The dashed arrow on the right is the image of the source cone. In particular, the image of the cone is not a face of the target.}
\end{figure}
\end{example}
\qed

The phenomenon is essentially ever-present, as the simplicity of the example might indicate. 

\begin{example}
Another large source of examples come from examining maps from rational curves to a toric pair $W$, see~\cite{R15b}. We assume that $W$ is smooth in order for it to fit into our present context. We may examine the morphism
\[
\mathsf K_\Gamma(W)\to \mathsf K_\Gamma(\mathsf A^r)\to \Mbar_{0,n}. 
\]
The first map is strict and smooth, and the composite is easily seen to be logarithmically smooth~\cite[Section~3]{R15b}. The main result of loc. cit. is that the first object above can be obtained from $\Mbar_{0,n}\times W$ by blowing up, and the map to $\Mbar_{0,n}$ is a projection. It is explained there how resulting blowup can be analyzed directly using the tropical combinatorics, and one immediately sees the non-flatness. 
\end{example}

\section{Logarithmic maps to a product}

The discussion in the previous section is meant to demonstrate that the failure of the product formula stems from the failure of 
\[
\mathsf K_\Gamma(\mathsf A^r)\to \Mbar_{g,n}
\]
to be flat with reduced fibers. The issue can be fixed by using the toroidal semistable reduction theorem due to Abramovich--Karu~\cite{AK00}, and done universally using Deligne--Mumford stacks in~\cite{Mol16}. Recall that a \textit{logarithmic modification} is a proper, birational, logarithmically \'etale morphism~\cite{AW}. Note that we allow the morphism to be non-representable, so generalized root constructions are permitted in the definition. 

\begin{theorem}[Abramovich--Karu]
Let $S\to B$ be a logarithmically smooth morphism of logarithmically smooth schemes. There exist logarithmic modifications $S^\dagger\to S$ and $B^\dagger\to B$ with an induced morphism
\[
S^\dagger\to B^\dagger
\]
is flat with reduced fibers. Moreover, $B^\dagger$ can be chosen to be smooth. 
\end{theorem}

When we write ``induced morphism'' above, we mean that the modification $S^\dagger$ is a further modification of the strict transform of $S$ along $B^\dagger\to B$, and therefore maps to $B^\dagger$.

We will now apply this to to the stabilization morphism to prove the main theorem. 

\subsection{Flattening moves} We now apply the theorem above to the stabilization morphism. We first replace $\mathsf K_\Gamma(\mathsf A^r)$ with a finite type open substack $\mathsf K^{\mathsf s}_\Gamma(\mathsf A^r)$ such that the morphism
\[
\mathsf K_\Gamma(W)\to \mathsf K_\Gamma(\mathsf A^r)
\]
factors through $\mathsf K^{\mathsf s}_\Gamma(\mathsf A^r)$. Explicitly, we do this by imposing that any subcurve that is formally contracted, i.e. that is decorated with vanishing curve class, is stable as a curve with its distinguished points (nodes and markings).

\begin{proposition}
There exist logarithmic modifications 
\[
\Mbar_{g,n}^\dagger\to \Mbar_{g,n} \ \ \ \ \mathsf K^{\mathsf s}_\Gamma(\mathsf A^r)^\dagger\to \mathsf K^{\mathsf s}_\Gamma(\mathsf A^r)
\]
such that there is an induced morphism 
\[
\mathsf K^{\mathsf s}_\Gamma(\mathsf A^r)^\dagger\to\Mbar_{g,n}^\dagger
\]
which is flat with reduced fibers. 
\end{proposition}

\begin{proof}
The result is an immediate consequence of the weak semistable reduction theorem, but since we work with Artin stacks with boundary divisors having monodromy, a sufficiently general statement is not obvious in the literature, primarily because the Artin fan construction fails to be functorial~\cite{AW}. We instead explain how to maneuver the problem so the existing statements apply. 

First, perform an iterated $\Mbar_{g,n}$ by blowing up each boundary stratum once in increasing order of dimension. The resulting space $\Mbar_{g,n}^{[1]}$ does not have self-intersecting boundary strata. Now, repeat the procedure again, blowing up each boundary stratum once, including the ones introduced in the first step, in increasing order of dimension. Call the resulting space $\Mbar_{g,n}^{[2]}$. It now has the property that every irreducible component of the boundary is smooth, and an intersection of any subset of irreducible boundary divisors is connected. 

Next, perform the following fine and saturated logarithmic base change:
\[
 \mathsf K^{\mathsf s}_\Gamma(\mathsf A^r)^{[\sim]} =  \left(\mathsf K^{\mathsf s}_\Gamma(\mathsf A^r)\times_{\Mbar_{g,n}} \Mbar_{g,n}^{[2]}\right)^{\mathsf{fs}}
\]
Concretely, this is done by taking the ideal sheaf of indeterminacy for the map $\Mbar_{g,n}^{[2]}\to\Mbar_{g,n}$, pulling it back as an ideal sheaf to $\mathsf K^{\mathsf s}_\Gamma(\mathsf A^r)$ and then blowing up. 

Now define $\mathsf K^{\mathsf s}_\Gamma(\mathsf A^r)^{[2]}$ by first performing toroidal resolution of singularities to make it normal crossings, and then applying the procedure already applied for $\Mbar_{g,n}$ to make it simple normal crossings, i.e. blow up strata in increasing order of dimension and then do it again. 

Now consider the induced morphism
\[
\mathsf K^{\mathsf s}_\Gamma(\mathsf A^r)^{[2]}\to \Mbar_{g,n}^{[2]}.
\]
This is a logarithmic morphism between simple normal crossings pairs. The results of Abramovich--Karu and Molcho may now be applied. Indeed, the source and target each have cone complexes and there is an induced map
\[
\Sigma^{\mathsf s}_\Gamma(\mathsf A^r)^{[2]}\to \Sigma_{g,n}^{[2]}
\]
of these cone complexes. The procedure of Molcho may now be applied without change to produce logarithmic modifications with the required properties~\cite[Section~2 {\it \&}~3]{Mol16}. Note that the surjectivity hypothesis in the statement of Molcho's theorem enters only for the canonicity of the procedure, not the existential statement.
\end{proof}

It will be useful to have some terminology; we maintain the notation above.

\begin{definition}
A \textit{semistabilization} of a morphism
\[
\mathsf K^{\mathsf s}_\Gamma(\mathsf A^r)\to \Mbar_{g,n}
\]
is a logarithmic modification of source and target such that there is an induced morphism 
\[
\mathsf K^{\mathsf s}_\Gamma(\mathsf A^r)^\dagger\to\Mbar_{g,n}^\dagger
\]
that is flat with reduced fibers.
\end{definition}

\subsection{Virtual strict transforms and the main theorem} Any logarithmic modification
\[
\mathsf K^{\mathsf s}_\Gamma(\mathsf A^r)^\dagger\to \mathsf K^{\mathsf s}_\Gamma(\mathsf A^r)
\]
induces one for the space of maps to $W$
\[
\mathsf K_\Gamma(W)^\dagger\to \mathsf K_\Gamma(W).
\]
We summarize the observations concerning the virtual structure of this stack from~\cite[Section~3.5]{R19} in the following proposition. 

\begin{proposition}
The morphism $\pi: \mathsf K_\Gamma(W)^\dagger\to \mathsf K^{\mathsf s}_\Gamma(\mathsf A^r)^\dagger$ has a relative perfect obstruction theory. The virtual pullback of the fundamental class defines a virtual fundamental class 
\[
\left[\mathsf K_\Gamma(W)^\dagger\right]^{\mathrm{vir}}:=\pi^{!}\left[\mathsf K^{\mathsf s}_\Gamma(\mathsf A^r)^\dagger\right]
\]
for $\mathsf K_\Gamma(W)^\dagger$. Moreover, pushforward along the modification
\[
\mathsf K_\Gamma(W)^\dagger\to \mathsf K_\Gamma(W)
\]
identifies virtual classes. 
\end{proposition}

The virtual class $[\mathsf K_\Gamma(W)^\dagger]^{\mathrm{vir}}$ functions as a \textit{virtual strict transform} of the virtual class on $\mathsf K_\Gamma(W)$ under the modification of $\mathsf K_\Gamma^{\mathsf{s}}(\mathsf A^r)$.

\subsection{Proof of Theorem~\ref{thm: product-formula}} We examine a product $Z$ of simple normal crossings pairs $X$ and $Y$. The divisors determine maps $X\to \mathsf A^{r_X}$ and $Y\to \mathsf A^{r_Y}$. We write these targets as $\mathsf A_X$ and $\mathsf A_Y$ for notational simplicity, and $\mathsf A_Z$ for the product.

Choose semistabilizations
\[
\mathsf K_\Gamma^{\mathsf{s}}(\mathsf A_X)^\dagger\to \Mbar_\Gamma, \ \ \ \mathsf K_\Gamma^{\mathsf{s}}(\mathsf A_Y)^\dagger\to \Mbar_\Gamma
\]
of the forgetful morphisms. Note that we may choose a single modification of $\Mbar_{g,n}$ that works for both forgetful morphisms. Indeed, once a semistabilization for $X$ is achieved, we may perform the fine and saturated base change of $\mathsf K^{\mathsf{s}}_\Gamma(\mathsf A_Y)$ to that modification and repeat the procedure. The flatness and reducedness properties are stable under base change so the semistabilization for $X$ can just be pulled back.

Now consider the following commutative diagram.
\[
\begin{tikzcd}
\mathsf K_\Gamma(Z)^\dagger\arrow{d}{\varphi}\arrow{r}{\vartheta} & \mathsf P(X\times Y) \arrow{r} \arrow{d} & \mathsf K_\Gamma(X)^\dagger\times \mathsf K_\Gamma(Y)^\dagger\arrow{d}{\phi} \\
\mathsf K^{\mathsf s}_\Gamma(\mathsf A_Z)^\dagger \arrow{r}[swap]{\nu} & \mathsf P(\mathsf A_X\times \mathsf A_Y) \arrow{r}[swap]{g} \arrow{d} & \mathsf K^{\mathsf s}_\Gamma(\mathsf A_X)^\dagger\times \mathsf K^{\mathsf s}_\Gamma(\mathsf A_Y)^\dagger\arrow{d}{\psi} \\
& \Mbar_\Gamma \arrow[swap]{r}{\Delta} & \Mbar_\Gamma\times\Mbar_\Gamma.
\end{tikzcd}
\]
The squares are defined by defined by fine and santurated logarithmic fiber product, but since $\psi$ is saturated and the top verticals are strict, this coincides with the fine and saturated logarithmic fiber product. We are now led to the main result, which asserts that after these logarithmic modifications, the virtual pullbacks in the above diagram are compatible. 

\begin{theorem}
There is equality of Chow homology classes on the space $\mathsf P(X\times Y)$:
\[
\vartheta_\star [\mathsf K_\Gamma(Z)^\dagger]^{\mathrm{vir}} = \Delta^! \left[ \mathsf K_\Gamma(X)^\dagger\times \mathsf K_\Gamma(Y)^\dagger\right]^{\mathrm{vir}}.
\]
\end{theorem}

\begin{proof}
We first inspect the vertical map $\psi$, which, on each factor forgets the data of the map and stabilizes if necessary. By construction the map is flat with reduced fibers.

The flatness of $\psi$ leads to the main conclusions necessary for the proof. As already noted, the flatness and reducedness of $\psi$ implies that the two squares on the right are Cartesian in the categories of fine and saturated logarithmic stacks, as well as in ordinary stacks. The morphisms $\varphi$ and $\phi$ each give relative obstruction theories and these are compatible, and the proof of compatibility is identical to~\cite[Proposition~6]{Beh97}. As noted above, the morphism $\psi$ is a flat morphism by construction, and therefore $g$ is the flat base change of $\Delta$. In turn, $\Delta$ is the diagonal morphism on a smooth Deligne--Mumford stack, and is therefore a local complete intersection morphism. Since the property of being local complete intersection is stable under flat base change, we see $g$ is also a local complete intersection morphism. The obstruction theories given by $\Delta$ and $g$ are therefore compatible~\cite[\href{https://stacks.math.columbia.edu/tag/069I}{Tag 069I}]{stacks-project}. 

Next, note that we have observed that both $\phi$ and $\varphi$ are equipped with perfect obstruction theories. They are compatible by an identical argument as used in~\cite[Proposition~8]{Beh97}.

Finally, we consider the morphism $\nu$. Since the stack $\mathsf P(\mathsf A_X\times \mathsf A_Y)$ is a fiber product in ordinary schemes, each point determines data
\[
(C,\widetilde C_1,\widetilde C_2, \widetilde C_1\to\mathsf A_X,\widetilde C_2\to \mathsf A_Y).
\]
That is families of logarithmic curves, together with two partial destabilizations of the curve, respectively equipped with logarithmic morphisms to $\mathsf A_X$ and to $\mathsf A_Y$.\footnote{Note that we have passed to a logarithmic modification, so the families considered are not necessarily minimal. The moduli functor is described in~\cite[Section~3.7]{R19}.}

The destabilizations are either contractions of rational tails or semistable components. Precisely, given a map to $\mathsf A_Z$, a genus $0$ component is destabilized in $\widetilde C_1$ if it has curve class decoration $0$ in the $X$ direction and is either $1$-pointed or $2$-pointed, and similarly for $\widetilde C_2$ and $Y$. 

Now since $\psi$ is logarithmically smooth and flat, the fiber product $\mathsf P(\mathsf A_X\times\mathsf A_Y)$ is logarithmically smooth and flat over $\Mbar_\Gamma$, and it follows that it is irreducible. It follows that the morphism $\nu$ has pure degree $1$. In addition, $\nu$ is proper, which can be seen by explicit verification of the valuative criterion, as follows. Suppose $R$ is a valuation ring and we fix and $R$-valued point $\mathsf P(\mathsf A_X\times\mathsf A_Y)$ and a $K$-valued lift along $\nu$. We must show that there is a unique $R$-valued lift along $\nu$. Since the source and target are both irreducible, it suffices to verify the valuative criterion when the $K$-valued point in $\mathsf P(\mathsf A_X\times\mathsf A_Y)$ lies in the locus of maps from smooth domains. The $R$-valued point determines curves $\mathscr C_1$ and $\mathscr C_2$ over $\spec R$ with the same generic fiber. The special fibers of each $\mathscr C_i$ potentially contains unstable components. The complement of these components gives rise to a birational map $\mathscr C_1\dashrightarrow \mathscr C_2$. The closure of the graph of this birational map in the product is a common destabilization of $\mathscr C_1$ and $\mathscr C_2$ which maps to $\mathsf A_Z$. Moreover, every component of the source of this map has a curve class $0$ decoration by at most one of the maps to the factors $\mathsf A_X$ or $\mathsf A_Y$. By adding decorations, the stability condition ensures that we obtain a unique lift to $\mathsf K_\Gamma^{\mathsf s}(\mathsf A_Z)$. 

%There are infinitely many $R$-valued limits of these data. However, there is a unique $R$-valued lift such every component maps isomorphically onto a component in either $\widetilde C_1$ or $\widetilde C_2$ (or both). 

Summarizing, we have shown that $\Delta$ and $g$ are local complete intersection morphisms with compatible obstruction theories, the maps $\varphi$ and $\phi$ have perfect obstruction theories and are compatible, and that $\nu$ is proper and of pure of degree $1$. The product formula now follows from a standard diagram chase, using the established compatibilities, as in~\cite{Beh97,LQ18}. The fact that $\nu$ is pure degree $1$ shows that
\[
\vartheta_\star[\mathsf K_\Gamma(Z)^\dagger]^{\mathrm{vir}} = \phi^![\mathsf P(\mathsf A_X\times \mathsf A_Y)]
\]
by an application of the Costello/Herr--Wise pushforward theorem, stated in~\cite[Theorem~5.0.1]{Cos06} with errata in~\cite[Theorem~1.1]{HW22}. Note that we have verified properness of $\nu$ above which shown to be a sufficient hypothesis for Costello's theorem in loc. cit.

We also have the equality
\[
\phi^![\mathsf P(\mathsf A_X\times \mathsf A_Y)]  = \phi^!g^![\mathsf K_\Gamma(X)^\dagger\times\mathsf K_\Gamma(Y)^\dagger]
= g^![\mathsf K_\Gamma(X)^\dagger\times\mathsf K_\Gamma(Y)^\dagger]^{\mathsf{vir}},
\]
by functoriality of virtual pullbacks~\cite[Theorem~4.1]{Mano12} applied to the top square on the right column of the diagram. Since $g^!$ and $\Delta^!$ have compatible obstruction theories, the define the same pullback morphism from the Chow groups of $\mathsf K_\Gamma(X)^\dagger\times\mathsf K_\Gamma(Y)^\dagger$ to $\mathsf P(X\times Y)$. The result follows. 

\end{proof}

\bibliographystyle{siam} 
\bibliography{Products}

\end{document}